\documentclass[11pt]{article}
\usepackage{amsfonts}
\usepackage{euscript,amstext}
\usepackage{mathrsfs}
\usepackage{indentfirst}
\usepackage{enumerate}
\usepackage{amssymb}
\usepackage{leftidx}
\usepackage{amsmath}
\usepackage{amsthm}
\usepackage{hyperref}
\usepackage{amscd}

\newtheorem{thm}{Theorem}[section]

\newtheorem{prop}[thm]{Proposition}
\newtheorem{lem}[thm]{Lemma}
\newtheorem{cor}[thm]{Corollary}

\theoremstyle{definition}
\newtheorem{defn}[thm]{Definition}

\theoremstyle{remark}
\newtheorem{rem}[thm]{Remark}
\newtheorem{bassum}[thm]{Basic Assumption}
\newtheorem{assum}[thm]{Assumption}

\numberwithin{equation}{section}

\DeclareMathOperator{\tr}{Tr}
\DeclareMathOperator{\SF}{sf}

\DeclareMathOperator{\ind}{Ind}
\DeclareMathOperator{\im}{Im}
\DeclareMathOperator{\spec}{Sp}
\DeclareMathOperator{\dom}{dom}

\DeclareMathOperator{\ch}{ch}

\title{Higher Spectral Flow for Dirac Operators with Local Boundary Conditions}
\author{
Jianqing YU\footnote{School of Mathematical Sciences,
University of Science and Technology of China,
96 Jinzhai Road, Hefei, Anhui 230026,
P. R. China. (jianqing@ustc.edu.cn)}
}
\date{}

\begin{document}

\maketitle

\begin{abstract}
We consider a gauge invariant one parameter family of families of fiberwise twisted Dirac
type operators on a fiberation with the typical fiber an even dimensional compact manifold
with boundary, i.e., a family $\{D_u\}, u\in [0,1]$ with $D_1=gD_0g^{-1}$ for a suitable
unitary automorphism $g$ of the twisted bundle.
Suppose all the operators $D_u$ are imposed with
a certain \emph{local elliptic} boundary condition $F$ and $D_{u,F}$ is the self-adjoint extension of $D_u$.
We establish a formula for the
higher spectral flow of $\{D_{u,F}\}$, $u\in[0,1]$. Our result
generalizes a recent result of Gorokhovsky and Lesch to the families case.
\end{abstract}

\section{Introduction}
The spectral flow for a curve of self-adjoint Fredholm operators
is introduced by Atiyah-Patodi-Singer \cite{MR0397799}, and is an integer that counts with
signs the number of eigenvalues passing through zero from the start of the curve to its end. (Throughout the paper a family always means
a continuous and sometimes smooth family, and a curve always means a one parameter family.) In \cite{MR3077920},
Prokhorova computed the spectral flow on a compact planar domain
for a curve which is invariant by a scalar gauge transformation
of Dirac type operators imposed with a certain local boundary condition.
In a recent paper \cite{gl}, based on the heat kernel approach,
Gorokhovsky and Lesch extended the results of \cite{MR3077920} to the case
of arbitrary compact even dimensional manifolds with boundary.

As a natural extension of the classical spectral flow to the families case, Dai and Zhang introduced
the higher spectral flow in \cite{MR1393386,MR1638328} for a curve of families of Dirac type operators
parameterized  by a compact space, which takes values in the $K$-group of the parameter space, with the help of
the notation of spectral sections introduced by Melrose-Piazza \cite{MR1472895}.
The higher spectral flow and spectral sections
have since then found other significant applications; see \cite{MR1601842},
\cite{MR1744589}, \cite{MR1979016}, \cite{MR3034288}, etc.

It is a natural question to generalize the result of Gorokhovsky-Lesch \cite{gl} to the families case.
We recall that Gorokhovsky and Lesch essentially used the Getzler formula \cite{MR1231957}
to reduce the problem to the computation of the heat kernel on a half-cylinder.
Our first observation is that by combining the arguments in \cite{gl} with
the heat kernel computations in \cite[Section 4]{MR1638328}, one can
obtain a families version of \cite[Theorem 3.3]{gl}, which expresses the Chern character of the higher spectral flow
of the families of Dirac type operators imposed with local boundary conditions in terms of that of the families of boundary operators.
However, the obtained formula fails to hold on the level of $K$-theory, since taking the Chern character
might kill some torsion elements involved in the higher spectral flow.

It is the purpose of the present paper to establish the desired formula in full generality
on $K$-theoretic level; see Theorem \ref{main}.
The general idea here is inspired by Dai-Zhang \cite{MR1447295} and is
to reduce the problem to the calculations on families of finite cylinders through
a splitting type formula of the higher spectral flow. The deduction of such a splitting formula
is essentially a modification of \cite[Section 3]{MR1447295}, which is based on the beautiful idea of
Bunke \cite{MR1331973} in his approach to the splitting formula for $\eta$-invariants.

As in \cite[Section 4]{MR1638328}, the higher spectral flow we consider here also gives an
interpretation of the index bundle of a Toeplitz type family associated to the family of Dirac type operators with a local boundary condition.
Thus, our result can be thought of as an odd analogue of a theorem of Freed \cite[Theorem  B]{MR1651419}
which expresses the families index of the Dirac type operators on odd dimensional fibers imposed with
a local boundary condition in terms of that of the boundary operators.

We indicate here that even though our results are stated for spin manifolds, both the statements
and proofs extend without difficulty to the more general case when there are smoothly varying $\mathbb{Z}_2$-graded
Hermitian Clifford modules over the fibers.

A brief outline of the paper is as follows. In Section \ref{sec2}, we recall the local boundary conditions, spectral sections and
the higher spectral flow, and introduce a Toeplitz type family whose index bundle can be interpreted as a higher spectral
flow. In Section \ref{sec3}, we state our main result
which extends \cite[Theorem 3.3]{gl} to the families case, and provide a proof.

\vspace{3mm}\textbf{Acknowledgements}\ \ This work was supported by the China Postdoctoral
Science Foundation (Grant No. 2014M551805). The author is indebted to
Professors Xianzhe Dai and Weiping Zhang for helpful suggestions.

\section{Local boundary conditions and higher spectral flow}\label{sec2}

\subsection{Local boundary conditions for a family of Dirac type operators}\label{sec2.1}
Let $\pi: M\rightarrow S$ be a smooth fibration with the typical fiber
an even dimensional compact manifold $Y$ with boundary $Z$. We assume $S$ is compact.
We fix a connection for this fibration
which amounts to a splitting
\begin{equation}\label{y3}
TM=TY\oplus T^HM,
\end{equation}
where $TY$ and $T^HM$ denote the vertical and the horizontal bundles, respectively.
We also have the identification $T^HM=\pi^*TS$.

Let $g^{TY}$ be a metric on $TY$. We choose a Riemannian metric $g^{TS}$ on $S$, and
let $g^{TM}$ be the metric on $TM$ corresponding to \eqref{y3} defined by the orthogonal sum
\begin{equation}\label{y25}
g^{TM}=g^{TY}\oplus \pi^*g^{TS}.
\end{equation}
Let $\nabla^{TM}$ be the Levi-Civita connection associated to $(TM,g^{TM})$. We denote by
$P^{TY}$ the orthogonal projection of $TM$ onto $TY$. Following Bismut \cite{MR813584},
let $\nabla^{TY}$ be the connection on $TY$ defined
by $\nabla^{TY}=P^{TY}\nabla^{TM}P^{TY}$. Then $\nabla^{TY}$ preserves $g^{TY}$, and is independent
of the choice of $g^{TS}$.

For any $b\in S$ and any vector bundle $\mathcal{E}$ over $M$, we set
\begin{equation*}
Y_b=\pi^{-1}(b),\ Z_b=\partial Y_b, \ \text{and}\ \mathcal{E}_b=\mathcal{E}|_{Y_b}.
\end{equation*}
Assuming that $\mathcal{E}$ carries a metric $g^\mathcal{E}$ and a connection $\nabla^\mathcal{E}$,
we denote by $\mathbf{H}^p(Y_b,\mathcal{E}_b)$, $p\geq 0$
the space of sections of $\mathcal{E}_b$ over $Y_b$ which lie in the $p$-th Sobolev space.
When $b$ runs through $S$, we get a Hilbert bundle
$\mathbf{H}^p(Y,\mathcal{E})$, $p\geq 0$ which will be called the $p$-th Sobolev bundle associated to $(g^\mathcal{E},\nabla^\mathcal{E})$.

We make the assumption that $TY$ is spin and carries a fixed spin structure.
Let $S(TY)$ be the Hermitian bundle of spinors associated to $(TY,g^{TY})$ with the $\mathbb{Z}_2$-grading $S(TY)=S_{+}(TY)\oplus S_{-}(TY)$.
Then $\nabla^{TY}$ induces naturally a Hermitian connection $\nabla^{S(TY)}$ on $S(TY)$ preserving the $\mathbb{Z}_2$-grading.

Let $E$ be a complex vector bundle over $M$, which carries a Hermitian metric $g^E$ and a Hermitian connection $\nabla^E$.
We denote by $c(\cdot)$ the Clifford action on $S(TY)$. Then it extends to an action $c(\cdot)\otimes {\rm id}|_E$ on $S(TY)\otimes E$,
which we still denote by $c(\cdot)$.
Let $\nabla^{S(TY)\otimes E}$ be the tensor product connection on $S(TY)\otimes E$ induced by $\nabla^{S(TY)}$ and $\nabla^E$.

For any $b\in S$, one has a canonically defined formally self-adjoint twisted Dirac operator which is given by
\begin{equation}
\begin{split}\label{d1}
D_b^{Y}&=\sum_{i=1}^{\dim Y}c(e_i)\nabla^{S(TY)\otimes E}_{e_i}:
\\
&\hspace{30pt}\Gamma(Y_b,S(TY)_b\otimes E_b)\longrightarrow \Gamma(Y_b,S(TY)_b\otimes E_b),
\end{split}
\end{equation}
where $\{e_i\}_{i=1}^{\dim Y}$ is an oriented orthonormal basis of $(TY_b, g^{TY}|_{Y_b})$.
We denote by $D^Y_{b,\pm}$ the restrictions of $D^Y_b$ on $\Gamma(Y_b,S_\pm(TY)_b\otimes E_b)$.
When $b$ runs through $S$, we obtain smooth families $D^Y=\{D^Y_b\}_{b\in S}$ and $D^Y_\pm=\{D^Y_{b,\pm}\}_{b\in S}$.

Let $\tfrac{\partial}{\partial r}$ denote the inward unit normal vector field perpendicular to $Z$.
Set $J=c(\tfrac{\partial}{\partial r})$.
Let $S(TZ)$ be the spinor bundle of $TZ$,  which carries a natural Hermitian metric induced by $g^{TZ}:=g^{TY}|_{TZ}$.
We identify $S(TZ)$ with $S_+(TY)|_{Z}$ isometrically.
Also, the Clifford action $\widetilde{c}(v)$ on $S(TZ)\otimes E|_{Z}$ can be specified by $-Jc(v)$ for $v\in TZ$.

Using the orthogonal decomposition induced by \eqref{y25},
\begin{equation}\label{yy5}
g^{TM}\big|_{T\partial M}=g^{TZ}\oplus (\pi|_{\partial M})^*g^{TS},
\end{equation}
we define a connection $\nabla^{TZ}$ on $TZ$,
which induces a Hermitian connection $\nabla^{S(TZ)}$ on $S(TZ)$.
Let $D^Z=\{D^Z_b\}_{b\in S}$ be the family of canonically defined twisted Dirac operators on
$S(TZ)\otimes E|_{Z}$ as in \eqref{d1}.

Let $D=\{D_b\}_{b\in S}$ be a smooth family of formally self-adjoint Dirac type operators in the sense of \cite[Definition 3.1]{MR1233386} on $\{Y_b\}_{b\in S}$.
We define the \emph{local boundary conditions} for $D$ as in \cite[Section 3.1]{gl}.

For a smooth family of invertible and self-adjoint
endomorphisms $F_b$ of $E|_{Z_b}$, $b\in S$, set (see \cite[(3.3)]{gl})
\begin{equation}
\label{lbd}
\dom (D_{b,F_b})=
\bigl\{s\in \mathbf{H}^1(Y_b,S(TY)_b\otimes E_b)\,\big|\, s_-|_{Z_b}=JF_bs_+|_{Z_b}\bigr\},
\end{equation}
where $s=s_+\oplus s_-$ splits with respect to the natural $\mathbb{Z}_2$-grading.

By \cite[Proposition~3.1]{gl}, we know that for every $b\in S$,
\begin{equation}\label{y22}
D_{b,F_b}=D_b\big|_{\dom (D_{b,F_b})}:\dom (D_{b,F_b})\longrightarrow \mathbf{H}^0(Y_b,S(TY)_b\otimes E_b)
\end{equation}
is a self-adjoint Fredholm operator. Furthermore, its
spectrum is discrete and each eigenvalue has finite multiplicity.
For the sake of simplicity, we will say $D_{F}=\{D_{b,F_b}\}_{b\in S}$ is the
realization of $D$ imposed with the local boundary condition $F=\{F_b\}_{b\in S}$.




\subsection{The higher spectral flow for families of Dirac operators with local boundary conditions}

Let $D=\{D_{b}\}_{b\in S}$ be a family of formally self-adjoint Dirac type operators. As in Section \ref{sec2.1},
we impose on $D$ a local boundary
condition $F=\{F_b\}_{b\in S}$ and denote its self-adjoint realization
by $D_{F}=\{D_{b,F_b}\}_{b\in S}$.

We first recall the notation of spectral sections
introduced by Melrose-Piazza \cite{MR1472895} in their study of the families index of
Dirac operators on manifolds with boundary.

\begin{defn}[{cf. \cite[Definition 1]{MR1472895}}]\label{defsps}
A spectral section for the family $D_F=\{ D_{b,F_b}\}_{b\in S}$ is a
family of self-adjoint projections $P_b$ on $\dom (D_{b,F_b})$
such that for some smooth function $R: S\rightarrow \mathbb{R}$ and every $b\in S$,
\begin{align}
D_{b,F_b}s=\lambda s \Rightarrow
\left\{
\begin{aligned}
P_b s&=s,\quad\text{if}\quad\lambda>R(b),
\\
P_b s&=0,\quad\text{if}\quad\lambda<-R(b).
\end{aligned}
\right.
\end{align}
\end{defn}

The existence of a spectral section of $D_F=\{ D_{b,F_b}\}_{b\in S}$ is guaranteed by
the following assumption; see \cite[Proposition 1]{MR1472895}.

\begin{bassum}\label{basic}
We assume that the canonical family $D^{Y}$ imposed with the local boundary condition $F$
has vanishing index bundle, i.e.,
\begin{equation}
\ind (D^Y_F)=0 \quad \text{in}\ K^1(S).
\end{equation}
\end{bassum}

In what follows, we always fixed a local boundary condition $F$, and make Basic Assumption \ref{basic}.

From Definition \ref{defsps}, one infers that any two spectral sections of ${D_F}$ differ by a family of
finite rank operators. Further, the generalized spectral section, which can be thought of
as a fiberwise compact perturbation of a spectral section, was introduced by Dai and Zhang \cite[Definition 2.1]{MR1447295} in their
study of the splitting of the families index.

\begin{defn}[{cf. \cite[Definition 2.1]{MR1447295}}]\label{defgsp}
A generalized spectral section for the family $D_F=\{ D_{b,F_b}\}_{b\in S}$ is a family of self-adjoint projections $Q=\{Q_b\}_{b\in S}$
such that $Q_b$ is a compact perturbation of $P_b$ for $b\in S$,
where $P=\{P_b\}_{b\in S}$ is a (in fact any) spectral section of $D_F$.
\end{defn}

Let $Q$, $Q'$ be two generalized spectral sections of $D_F$. For $b\in S$, set
\begin{equation}
T_b(Q,Q')=Q'_{b}Q_{b}:\im(Q_b)\rightarrow \im (Q'_b).
\end{equation}
Since $Q$ differs by a family of compact operators from $Q'$,
by \cite[Proposition 3.2]{MR1262254}, $T(Q,Q')=\{T_b(Q,Q')\}_{b\in S}$
defines a family of Fredholm operators. Thus according to Atiyah-Singer \cite{MR0279833},
it determines an element
\begin{equation}
[Q-Q']:=\ind T(Q,Q')\in K(S).
\end{equation}

We now recall the definition of the higher spectral flow of Dai-Zhang (cf. \cite[Definition 1.5, Theorem 1.11]{MR1638328})
for families of Dirac type operators imposed with local boundary conditions.

Consider a curve of families of formally self-adjoint Dirac operators $D_u=\{D_{u,b}\}_{b\in S}$, $u\in [0,1]$.
We impose on $D_u$, $u\in [0,1]$ the local boundary condition $F$ and
denote its realization by $D_{u,F}=\{D_{u,b,F_b}\}_{b\in S}$.

Let $Q_0$, $Q_1$ be spectral sections of $D_{0,F}$, $D_{1,F}$, respectively. If we consider the total
family $\widetilde{D_{F}}=\{D_{u,b,F_b}\}_{(u,b)\in [0,1]\times S}$, then there exists a total generalized spectral section
$\widetilde{P}=\{P_{u,b}\}_{(u,b)\in [0,1]\times S}$. Let $P_u$ be the restriction of $\widetilde{P}$ over $\{u\}\times S$.

\begin{defn}\label{defhsf}
The higher spectral flow $\SF\{(D_{0,F},Q_0),(D_{1,F},Q_1)\}$ between the pairs $(D_{0,F},Q_0)$ and $(D_{1,F},Q_1)$ is
an element in $K(S)$ defined by
\begin{equation}
\SF\{(D_{0,F},Q_0),(D_{1,F},Q_1)\}=[Q_1-P_1]-[Q_0-P_0]\in K(S).
\end{equation}
\end{defn}

\begin{rem}
It can be shown by some elementary arguments concerning Fredholm families
that the definition is independent of the choice of
the total generalized section $\widetilde{P}$; see \cite[pp. 440-441]{MR1638328} and \cite[Section 3]{MR1262254}.
\end{rem}

\subsection{Toeplitz families and higher spectral flows}\label{sec2.4}

We continue in the notations of the previous subsection.

Let $D=\{D_{b}\}_{b\in S}$ be a family of formally self-adjoint Dirac type operators
imposed with the local boundary
condition $F=\{F_b\}_{b\in S}$.

Under Basic Assumption \ref{basic}, we can pick a spectral section $P=\{P_b\}_{b\in S}$ of the family
$D_{F}=\{D_{b,F_b}\}_{b\in S}$.

Let $g: M\rightarrow U(N)$ be a smooth map from $M$ to the unitary group $U(N)$.
Then $g$ can be viewed as an automorphism of the
trivial complex vector bundle $\mathbb{C}^N|_M$ over $M$. For any $b\in S$, $g$
induces an automorphism of  $\mathbb{C}^N|_{Y_b}$. Thus for any $b\in S$, $g$ induces
a bounded map $g_b$ from
\begin{equation*}
\mathbf{H}^0(Y_b,S(TY)_b\otimes E_b\otimes \mathbb{C}^N)
\end{equation*}
to itself by acting as an identity on $\mathbf{H}^0(Y_b,S(TY)_b\otimes E_b)$.

Also it is clear that $P_b$ induces a map on $\mathbf{H}^0(Y_b,S(TY)_b\otimes E_b\otimes \mathbb{C}^N)$ by
acting as an identity on sections of $\mathbb{C}^N$. We still note this map by $P_b$.

\begin{defn}[{Compare with \cite[Definition 4.1]{MR1638328}}]
The Toeplitz family $T_{g,F}=\{T_{g,F,b}\}_{b\in S}$ is a family of Toeplitz operators defined by
\begin{equation}
\begin{split}
T_{g,F,b}&=P_bg_bP_b:P_b\,\mathbf{H}^0(Y_b,S(TY)_b\otimes E_b\otimes \mathbb{C}^N)
\\
&\hspace{80pt}\longrightarrow P_b\,\mathbf{H}^0(Y_b,S(TY)_b\otimes E_b\otimes \mathbb{C}^N).
\end{split}
\end{equation}
\end{defn}

We define the conjugation $D_g=\{D_{g,b}\}_{b\in S}$ of $D$ by $g$ as in \cite[(4.7)]{MR1638328}, i.e., set
for any $b\in S$,
\begin{equation}\label{conju}
D_{g,b}=g_bD_{b}g_b^{-1}.
\end{equation}
Clearly, $D_g$ is also a family of formally self-adjoint Dirac type operators.
We denote by $D_{g,F}=\{D_{g,b,F_b}\}_{b\in S}$ the realization of $D_g$ imposed with the local boundary condition $F$.
By Definition \ref{defsps}, $gPg^{-1}=\{g_bP_bg_b^{-1}\}_{b\in S}$ gives a spectral section of $D_{g,F}$.

By \eqref{conju}, $D_{g,b,F_b}-D_{b,F_b}$
is $D_{b,F_b}$-compact for $b\in S$; see
\cite[Chapter {\uppercase\expandafter{\romannumeral 4}}.1.3]{MR1335452} for the definition of relatively compactness.
It follows from \cite[Corollary 3.5]{MR2114489} that $gPg^{-1}-P$ is a family of compact operators.

Therefore, $T_{g,F}$ is a family of Fredholm operators\footnote{Indeed,
\begin{equation*}
T_{g,F}T_{g^{-1},F}-P=P(gPg^{-1}-P)P,
\end{equation*}
which clearly is a family of compact operators.}.
As a consequence, it defines an index bundle
\begin{equation}\label{indto}
\ind (T_{g,F})=[gPg^{-1}-P]\in K(S).
\end{equation}
As in the proof of \cite[Proposition 4.3(i)]{MR1638328},
we know that the right hand side of \eqref{indto} does not depend on the choice of $P$.

As in \cite[Section 4]{MR1638328}, the index bundle of the Toeplitz family can be interpreted via
higher spectral flow.

For any $0\leq u\leq 1$, set
\begin{equation}\label{toep}
D_{u,b}=(1-u)D_{b}+uD_{g,b}.
\end{equation}
Then we get a curve $D_u=\{D_{u,b}\}_{b\in S}$, $u\in [0,1]$ of families of formally self-adjoint
Dirac type operators.
Imposing the local boundary condition $F$ on the curve \eqref{toep},
we get a curve
\begin{equation*}
D_{u,F}=\{D_{u,b,F_b}\}_{b\in S}, u \in [0,1]
\end{equation*}
of families of self-adjoint Fredholm operators.
By Definition \ref{defhsf}, we have a higher spectral flow
\begin{equation}
\SF\big\{(D_F,P),(D_{g,F},gPg^{-1})\big\}\in K(S).
\end{equation}

Since $D_{u,b,F_b}$ is $D_{b,F_b}$-compact for $b\in S$ and $u\in [0,1]$, by \cite[Corollary 3.5]{MR2114489},
$P$ is a generalized spectral section
in the sense of Definition \ref{defgsp} for the total family $\widetilde{D_F}=\{D_{u,b,F_b}\}_{(u,b)\in [0,1]\times S}$.
Thus, by Definition \ref{defhsf}, we have
\begin{equation}\label{sf}
\SF\big\{(D_F,P),(D_{g,F},gPg^{-1})\big\}=[gPg^{-1}-P].
\end{equation}
As explained below \eqref{indto},
the right hand side of \eqref{sf} does not depend on
the choice of $P$. We will denote it by $\SF\{D_F,D_{g,F}\}$.

\begin{thm}[{Compare with \cite[Theorem 4.4]{MR1638328}}]
The following identity holds in $K(S)$,
\begin{equation}
\ind(T_{g,F})=\SF\{D_F,D_{g,F}\}.
\end{equation}
\end{thm}

The next proposition concerning the homotopy invariance of the higher spectral flow
is inspired by \cite[Proposition 2.1]{gl}.

Let $D_s$, $s\in [0,1]$ be a curve of families of formally self-adjoint Dirac type operators.
Let $g_t:M\rightarrow U(N)$, $t\in [0,1]$ be a curve of automorphisms of $\mathbb{C}^N|_M$.
For $s,t\in [0,1]$, we denote by $D_{s,g_t}$ the conjugation of $D_s$ by $g_t$ as
in \eqref{conju}.

Let $E_{+}$ and $E_{-}$ denote the positive and
the negative eigenbundles of $F$, respectively.
Set $\widetilde{F}={\rm id}|_{E_+}\oplus (-{\rm id}|_{E_-})$.
Then
\begin{equation}\label{deform}
F_v=v \widetilde{F}+(1-v)F, \ v\in [0,1]
\end{equation}
defines a curve of families of invertible and self-adjoint endomorphisms of $\{E|_{Z_b}\}_{b\in S}$, which
amounts to a curve of local boundary conditions.

For $s,t,v\in [0,1]$, we impose the local boundary condition $F_v$ on the curve
\begin{equation}
D_{s,u}(t)=(1-u)D_{s}+uD_{s,g_t},\ u\in [0,1],
\end{equation}
and get a higher spectral flow $\SF\{D_{s,F_v},D_{s,g_t,F_v}\}$.

\begin{prop}\label{hoto}
The following identity holds in $K(S)$,
\begin{equation}\label{homotopy}
\SF\{D_{0,F},D_{0,g_0,F}\}=\SF\{D_{1,F},D_{1,g_1,F}\}=\SF\{D_{1,\widetilde{F}},D_{1,g_1,\widetilde{F}}\}.
\end{equation}
\end{prop}
\begin{proof}
By \cite[Theorem 7.16]{MR2536846}, we know that the Fredholm families
\begin{equation}
(u,s,t,v)\longmapsto D_{s,u,F_v}(t)=(1-u)D_{s,F_v}+uD_{s,g_t,F_v}
\end{equation}
are (fiberwise) continuous in the graph topology. Thus, we can easily modify the
argument in the proof of  \cite[Proposition 2.1]{gl} where the case of single operators
is treated to the families case, and deduce \eqref{homotopy}.
\end{proof}

With the help of the preceding proposition we can now prove the following
families extension of an intermediary result that appears in the second proof of \cite[Corollary 3.4]{gl}.

We write $F>0$, if for any $b\in S$, all the eigenvalues of $F_b$ are positive.

\begin{prop}\label{y7}
If $F>0$ or $-F>0$, then
\begin{equation}\label{y2}
\SF\big\{D_F,D_{g,F}\big\}=0 \quad \text{in}\ K(S).
\end{equation}
\end{prop}
\begin{proof}
Let $\gamma$ be the grading operator on $S(TY)\otimes E$ such that
\begin{equation*}
\gamma|_{S_+(TY)\otimes E}={\rm id},\ \ \gamma|_{S_-(TY)\otimes E}=-{\rm id}.
\end{equation*}

We claim the fiberwise invertibility of the curve $(D+\gamma)_{u,F}$, $u\in [0,1]$ obtained by imposing the local boundary condition
$F$ on the curve
\begin{equation}
\begin{split}
(D+\gamma)_u
&=(1-u)(D+\gamma)+ug(D+\gamma)g^{-1}
\\
&=D_u+\gamma,\hspace{10pt} u\in [0,1].
\end{split}
\end{equation}
It follows the claim that the associated higher spectral flow vanishes.
Due to Proposition \ref{hoto}, \eqref{y2} is proved.

We now turn to the proof of our claim.

For $b\in S$, $s=s_+\oplus s_-\in \ker\big((D+\gamma)_{u,b,F_b}\big)$, we have
\begin{equation}\label{y24}
(D_{u,b}+\gamma)s_+=0,\ \ (D_{u,b}+\gamma)s_-=0, \ \ s_-|_{Z_b}=JF_bs_+|_{Z_b}.
\end{equation}
Applying Green's formula \cite[Proposition 3.4(b)]{MR1233386}, we deduce that
\begin{equation}
\begin{split}\label{y23}
&\big\langle(D_{u,b}+\gamma)(s_+\oplus 0),0\oplus s_-\big\rangle_{L^2(Y_b)}
\\
&\hspace{40pt}-\big\langle s_+\oplus 0,(D_{u,b}+\gamma)(0\oplus s_-)\big\rangle_{L^2(Y_b)}
\\
&=-\big\langle Js_+|_{Z_b},s_-|_{Z_b}\big\rangle_{L^2(Z_b)},
\end{split}
\end{equation}
where by $\langle\cdot,\cdot\rangle_{L^2(Y_b)}$ and $\langle\cdot,\cdot\rangle_{L^2(Z_b)}$
we mean the $L^2$ inner products naturally defined on $Y_b$ and $Z_b$, respectively.
Combining \eqref{y24} with \eqref{y23}, we infer that $s_+|_{Z_b}=s_-|_{Z_b}=0$,
since $F_b$ (or $-F_b$) is positive.

Now that $s|_{Z_b}=0$, it follows from Green's formula that
\begin{equation}
\begin{split}\label{y44}
&\big\langle D_{u,b}s,\gamma s\big\rangle_{L^2(Y_b)}+\big\langle\gamma s,D_{u,b}s\big\rangle_{L^2(Y_b)}
\\
&\hspace{30pt}=\big\langle \big(D_{u,b}\gamma + \gamma D_{u,b}\big)s, s\big\rangle_{L^2(Y_b)},
\end{split}
\end{equation}
which is equal to zero, since $D_{u,b}$ obviously anticommutes with $\gamma$.
On the other hand, it is trivial to see that
\begin{equation}\label{y45}
\big\langle(D_{u,b}+\gamma)s,(D_{u,b}+\gamma)s\big\rangle_{L^2(Y_b)}=0.
\end{equation}
By substituting \eqref{y44} into \eqref{y45}, we obtain
\begin{equation}
\big\langle D_{u,b}s,D_{u,b}s\big\rangle_{L^2(Y_b)}+\big\langle\gamma s,\gamma s\big\rangle_{L^2(Y_b)}=0,
\end{equation}
which implies $s|_{Y_b}=0$ as claimed.
\end{proof}

\section{Main result and its Proof}\label{sec3}

\subsection{Main Result}\label{sec3.1}

We first recall some notations that we use in last section.

Let $\pi: M\rightarrow S$ be a smooth fibration with the typical fiber
an even dimensional compact spin manifold $Y$ with boundary $Z$ and $S$ compact.

Let $g^{TY}$ be a metric on $TY$, and $(E,g^E)$ be a complex Hermitian
vector bundle over $M$ endowed with a Hermitian connection $\nabla^E$.

Recall that $D^Y=\{D^Y_b\}_{b\in S}$ is
a family of fiberwise twisted Dirac operators
on $S(TY)\otimes E$ over $\{Y_b\}_{b\in S}$ defined in \eqref{d1}.
Let $D^Y_g$ be the conjugation of $D^Y$ by $g$
defined as in \eqref{conju} for a smooth map $g: M\rightarrow U(N)$.

Given a family $F=\{F_b\}_{b\in S}$ of invertible and self-adjoint endomorphisms of $\{E|_{Z_b}\}_{b\in S}$ which amounts to
a local boundary condition, as in \eqref{lbd} and \eqref{y22}, we obtain the self-adjoint
realizations $D^Y_F$ of $D^Y$ and $D^Y_{g,F}$ of $D^Y_g$. Furthermore, we have a higher spectral flow
$\SF\{D^Y_F,D^Y_{g,F}\}$ as in Section \ref{sec2.4}, provided that
we make Basic Assumption \ref{basic}.

Let $\Pi_{+}$ and $\Pi_{-}$ be the bundle projections on to the positive eigenbundle $E_+$
and the negative eigenbundle $E_-$ of $F$, respectively. Set
\begin{equation}
g^{E_\pm}=g^E|_{E_\pm}\,, \
\nabla^{E_+}=\Pi_{+}\nabla^E|_Z\,\Pi_{+}\,, \
\nabla^{E_-}=\Pi_{-}\nabla^E|_Z\,\Pi_{-}\,. \
\end{equation}
Then $\nabla^{E_\pm}$ preserve the Hermitian metrics $g^{E_\pm}$ on $E_{\pm}$, respectively.

Recall that the connection $\nabla^{S(TZ)}$ on $TZ$ is introduced below \eqref{yy5} and
$D^Z$ is a family of self-adjoint twisted Dirac operators on $S(TZ)\otimes E$
defined in terms of $\nabla^{S(TZ)}$ and $\nabla^{E}$. Let $B^Z_{\pm}$ be
the families of self-adjoint twisted Dirac operators on $S(TZ)\otimes E_\pm$
defined using $\nabla^{S(TZ)}$ and $\nabla^{E_\pm}$ in a similar way to $D^Z$. Set
\begin{equation}\label{bdeq}
B^Z=B^Z_{+}\oplus B^Z_{-}.
\end{equation}
It is straightforward to show that $B^Z$ differs by a family of
differential operators of zeroth order from $D^{Z}$.

We define the conjugations $B^Z_{\pm,g|_{\partial M}}$ and $D^Z_{g|_{\partial M}}$ as in \eqref{conju}. Set
\begin{equation}
B^Z_{\pm,g|_{\partial M}}=g|_{\partial M}B^Z_{\pm}g^{-1}|_{\partial M},\
D^Z_{g|_{\partial M}}=g|_{\partial M}D^Z g^{-1}|_{\partial M}.
\end{equation}
As will be explained in Section \ref{sec3.3}, under
Basic Assumption \ref{basic}, the spectral sections for $B^Z_{\pm}$, $D^Z$ always exist.
By \cite[Definition 1.5]{MR1638328}, we have higher spectral flows
\begin{equation*}
\SF\{B^Z_{+},B^Z_{+,g|_{\partial M}}\},
\ \SF\{B^Z_{-},B^Z_{-,g|_{\partial M}}\}
\ \text{and}\ \SF\{D^Z,D^Z_{g|_{\partial M}}\}
\end{equation*}
associated to the curves
\begin{align}\label{y14}
B^Z_{\pm,u}&=(1-u)B^Z_\pm+uB^Z_{\pm,g|_{\partial M}}\,,\ u\in [0,1],
\\
D^Z_u&=(1-u)D^Z+uD^Z_{g|_{\partial M}}\,, \ u\in [0,1],
\end{align}
of families of self-adjoint twisted Dirac type operators over $\{Z_b\}_{b\in S}$.

We are now in a position to state our main result which is a generalization of
\cite[Theorems 3.3]{gl} to the families case.

\begin{thm}\label{main}
The following identity holds in $K(S)$,
\begin{equation}\label{maeq}
\SF\{D^Y_F,D^Y_{g,F}\}=\SF\{B^Z_{+},B^Z_{+,g|_{\partial M}}\}=-\SF\{B^Z_{-},B^Z_{-,g|_{\partial M}}\}.
\end{equation}
\end{thm}
\begin{proof}
The proof will be given in Sections \ref{sec3.2}-\ref{sec3.3}.
\end{proof}

\begin{cor}[Cobordism invariance of higher spectral flow] One has,
\begin{equation}\label{coreq}
\SF\{D^Z,D^Z_{g|_{\partial M}}\}=0 \quad \text{in}\ K(S).
\end{equation}
\end{cor}
\begin{proof}
Since $D^{Z}$ and $B$ have the same principal symbol,
\eqref{coreq} is a direct consequence of \eqref{maeq} and \cite[Proposition 4.3(ii)]{MR1638328}.
\end{proof}

We now indicate some characteristic forms that we will use.

Recall that $\nabla^{TZ}$ is a connection on $TZ$ defined below \eqref{yy5}. Let $R^{TZ}$ be the
curvature of $\nabla^{TZ}$. We define
$\widehat{A}(TZ,\nabla^{TZ})$ as in \cite[(1.19)]{MR1864735},
\begin{equation}
\widehat{A}(TZ,\nabla^{TZ})=\det
\left(\left(\frac{\frac{\sqrt{-1}}{4\pi}R^{TZ}}{\sinh (\frac{\sqrt{-1}}{4\pi}R^{TZ})}\right)^{1/2}\,\right).
\end{equation}

Let $R^{E_+}$ be the curvature of $\nabla^{E_+}$, and $\ch(E_+,\nabla^{E_+})$ be the Chern character form defined
as in \cite[(1.22)]{MR1864735},
\begin{equation}
\ch(E_+,\nabla^{E_+})=\tr\Big[\exp\Big(\frac{\sqrt{-1}}{2\pi}R^{E_+}\Big)\Big].
\end{equation}
We also define $\ch(E_-,\nabla^{E_-})$ in a similar way.

Let $d$ be the trivial Hermitian connection on $\mathbb{C}^N|_{M}$.
Let $\ch(g,d)$ denote the odd Chern character form of $g$ defined as in \cite[(1.50)]{MR1864735} (see also \cite[Section 1]{MR1231957}),
\begin{equation}
\ch(g,d)=\sum_{n=0}^{\infty}\left(\frac{1}{2\pi\sqrt{-1}}\right)^{\frac{n+1}{2}}\frac{n!}{(2n+1)!}\tr\big[(g^{-1}dg)^n\big].
\end{equation}

\begin{cor}\label{cormain}
The following identity holds in $H^*(S)$,
\begin{equation}
\begin{split}\label{yy4}
\ch(\SF\{D^Y_F,D^Y_{g,F}\})
&=-\int_Z\widehat{A}(TZ,\nabla^{TZ})\ch(E_+,\nabla^{E_+})\ch(g,d)
\\
&=\int_Z\widehat{A}(TZ,\nabla^{TZ})\ch(E_-,\nabla^{E_-})\ch(g,d).
\end{split}
\end{equation}
\end{cor}
\begin{proof}
Combining \cite[Theorem 4.9]{MR1638328} with Theorem \ref{main}, we get \eqref{yy4}.
\end{proof}

\subsection{A reduction of the problem to the product case}

We show that we can reduce the computation of the higher spectral flow to the situation where all
objects (metrics, connections, etc.) have product structure decompositions near the boundary.

In the rest of this paper, we will suppress the subscript $b$ for $b\in S$. All the constructions below are
understood fiberwise unless stated otherwise. Also, while being fiberwise, all the procedures will
be continuous with respect to $b\in S$, as will be clear from the context.

Let ${\rm inj}^Y$ denote the injectivity radius of $(Y,g^{TY})$. Pick $0<\epsilon<\frac{1}{16}\,{\rm inj}^Y$, and fix it once for all.
We use the inward geodesic flow to identify a neighborhood of $Z$
in $Y$ with the collar $Z \times [0,16\epsilon)$.

For $0\leq\mu<\nu<16\epsilon$, we set
\begin{equation}
Z_{[\mu,\nu]}=Z\times[\mu,\nu], \quad Y(\nu)=Y-Z_{[0,\nu]}.
\end{equation}
Let $\psi$ be a smooth function on $Y$ such that
\begin{equation}
\psi|_{Z_{[0,12\epsilon]}}=1,\quad 0\leq\psi|_{[12\epsilon,14\epsilon]}\leq 1,\quad \psi|_{Y(14\epsilon)}=0.
\end{equation}

Let ${\rm pr}_1:Z_{[0,15\epsilon]}\rightarrow Z$ denote the obvious projection onto the first factor.
Recall that $g^{TZ}=g^{TY}|_{TZ}$ and that $\tfrac{\partial}{\partial r}$ is
the inward unit normal vector field perpendicular to $Z$. Set
\begin{equation}\label{y46}
g'^{\,TY}=\psi\,({\rm pr}^*_1g^{TZ}\oplus dr^2)+(1-\psi)g^{TY}.
\end{equation}
We infer from \eqref{y46} that
the connection $\nabla'^{\,TY}$ on $TY$ defined through $g'^{\,TY}$ as in Section \ref{sec2.1}
has a fiberwise product structure decomposition over $Z_{[0,12\epsilon]}$, i.e.,
\begin{equation}
\nabla'^{\,TY}|_{Z_{[0,12\epsilon]}}={\rm pr}^*_1(\nabla'^{\,TY}|_Z).
\end{equation}

We trivialize $E|_{Z_{[0,15\epsilon]}}$ over $Z_{[0,15\epsilon]}$ by parallel
transport with respect to $\nabla^E$
along the geodesics perpendicular to $Z$. Set
\begin{equation}
g'^E=\psi\,{\rm pr}^*_1(g^{E}|_Z)+(1-\psi)g^{E}.
\end{equation}
Take any Hermitian connection $\nabla_1^{E}$ on $(E,g'^E)$, and set
\begin{equation}
\nabla'^E=\psi\,{\rm pr}^*_1(\nabla^{E_+}\oplus\nabla^{E_-})+(1-\psi)\nabla_1^{E}.
\end{equation}
Then $\nabla'^E$ preserves $g'^E$.

Let $\mathbf{H}^0(Y,S(TY)\otimes E)$ and $\mathbf{H}'^{\,0}(Y,S(TY)\otimes E)$ be the $0$-th Sobolev bundles of $S(TY)\otimes E$ over $Y$
associated to the metrics induced by $(g^{TY},g^E)$ and $(g'^{\,TY},g'^E)$, respectively.

By \cite[Lemma 2.1]{MR2536846}, we can construct a family $\Psi=\{\Psi_b\}_{b\in S}$ of invertible even-parity endomorphisms of $S(TY)\otimes E$
such that $\Psi|_{Z}={\rm id}$,
which extends to a family of isometries
\begin{equation}
\Psi: \mathbf{H}^0(Y,S(TY)\otimes E)\longrightarrow  \mathbf{H}'^{\,0}(Y,S(TY)\otimes E).
\end{equation}

Using $g'^{\,TY}$, $g'^E$ and $\nabla'^E$, we define
a family of fiberwise twisted Dirac operators $D'^Y=\{D_b'^Y\}_{b\in S}$ on $S(TY)\otimes E$ as in \eqref{d1}. Set
\begin{equation*}
D'=\Psi^{-1}\circ D'^Y\circ\Psi.
\end{equation*}
Then from \cite[Lemma 2.1]{MR2536846}, we see that $D'$ is a family of
formally self-adjoint Dirac type operators on $S(TY)\otimes E$ endowed with the metric induced by $g^{TY}$ and $g^E$.
Furthermore, since $\Psi$ is fiberwise unitary and $\Psi|_{Z}={\rm id}$, combining with
Proposition \ref{hoto}, we deduce that
\begin{equation}\label{y41}
\SF\{D'^Y_F,D'^Y_{g,F}\}=\SF\{D'_F,D'_{g,F}\}=\SF\{D^Y_F,D^Y_{g,F}\}\quad \text{in}\ K(S).
\end{equation}

Let $A$ be a smooth map from $Z_{[0,15\epsilon]}$ to $\mathfrak{u}(N)$, the Lie algebra of $U(N)$, which is such that
\begin{equation}
g(z,r)=g(z,0)\exp A(z,r),\ \text{for}\ (z,r)\in Z_{[0,15\epsilon]}.
\end{equation}
Then $\psi A$ clearly extends to $Y$, and
\begin{equation}
g_t=g\exp(-t\psi A),\ t\in [0,1]
\end{equation}
defines a curve of automorphisms of $\mathbb{C}^N|_M$ such that
\begin{equation}\label{y13}
g_0=g, \ \ g_1\big|_{Z_{[0,12\epsilon]}}={\rm pr}_1^*(g|_{Z}).
\end{equation}
It follows from Proposition \ref{hoto} that
\begin{equation}\label{sp2}
\SF\{D'^Y_F,D'^Y_{g,F}\}=\SF\{D'^Y_{F},D'^Y_{g_1,F}\}\quad \text{in}\ K(S).
\end{equation}

Combining \eqref{y41} with \eqref{sp2}, we see that
\begin{equation}
\SF\{D^Y_F,D^Y_{g,F}\}=\SF\{D'^Y_{F},D'^Y_{g_1,F}\} \quad \text{in}\ K(S).
\end{equation}

From the preceding discussion, it suffice to work in the case
where all objects involved have product structure decompositions near the boundary.
Specifically, we will always make the following assumption in the remainder of this paper.

\begin{assum}\label{y42}
The following equations hold on $Z_{[0,12\epsilon]}$,
\begin{equation}
\begin{split}\label{y43}
g^{TY}|_{Z_{[0,12\epsilon]}}&={\rm pr}^*_1g^{TZ}\oplus dr^2, \ \ g^E|_{Z_{[0,12\epsilon]}}={\rm pr}^*_1(g^{E}|_Z),
\\
\nabla^E|_{Z_{[0,12\epsilon]}}&={\rm pr}^*_1(\nabla^{E_+}\oplus\nabla^{E_-}),\ \ g\big|_{Z_{[0,12\epsilon]}}={\rm pr}_1^*(g|_{Z}).
\end{split}
\end{equation}
\end{assum}

\subsection{A splitting formula for the higher spectral flow}\label{sec3.2}

We use $J=c(\tfrac{\partial}{\partial r})$ to identify $S_-(TY)$ with $S_+(TY)$ over $Z_{[0,12\epsilon]}$ via the map
\begin{equation}\label{identi}
S_-(TY)|_{Z_{[0,12\epsilon]}}\simeq S_+(TY)|_{Z_{[0,12\epsilon]}}, \ \  s_- \longmapsto -Js_-.
\end{equation}
Then the local boundary condition defined by $F$ in \eqref{lbd} takes the form under identification \eqref{identi}:
\begin{equation}
\begin{split}
\bigl\{s_+\oplus s_-\,\big|\,&s_\pm\in \mathbf{H}^{1}(Y_b,S_+(TY)_b\otimes E_b), s_-|_{Z_b}=F_bs_+|_{Z_b}\bigr\},\ b\in S.
\end{split}
\end{equation}

Recall that $S(TZ)$ is identified with $S_+(TY)|_Z$ in Section \ref{sec2.1}.
It follows from Assumption \ref{y42} that
$B^{Z}$ can be thought of as a family of differential operators
acting on $(S_+(TY)\otimes E)|_{Z_{[0,12\epsilon]}}$, which is constant along the geodesics perpendicular to $Z$.

By Assumption \ref{y42} and \eqref{identi}, we will regard
the restriction of $D^Y$ on $Z_{[0,12\epsilon]}$ as a family of operators which acts on
\begin{equation*}
\big(S_+(TY)\otimes E\oplus S_+(TY)\otimes E\big)\big|_{Z_{[0,12\epsilon]}}
\end{equation*}
in the following way:
\begin{equation}\label{bdeq2}
\begin{bmatrix}0 &D^Y_-J
\\
-JD^Y_+ & 0
\end{bmatrix}
=\begin{bmatrix}0 &-\tfrac{\partial}{\partial r}+B^Z
\\
\tfrac{\partial}{\partial r}+B^Z & 0
\end{bmatrix}.
\end{equation}

For $0\leq\mu<\nu\leq 12\epsilon$, as at the beginning of Section \ref{sec2.1},
given any vector bundle $\mathcal{E}$ endowed with a metric $g^\mathcal{E}$ and a connection $\nabla^{\mathcal{E}}$ over $M$,
we denote by $\mathbf{H}^p(Z,\mathcal{E})$, $\mathbf{H}^p(Z_{[\mu,\nu]},\mathcal{E})$
and $\mathbf{H}^p(Y(\nu),\mathcal{E})$, $p\geq 0$ the $p$-th Sobolev bundles
associated to the restrictions $(\mathcal{E},g^{\mathcal{E}},\nabla^{\mathcal{E}})|_Z$,
$(\mathcal{E},g^{\mathcal{E}},\nabla^{\mathcal{E}})|_{Z_{[\mu,\nu]}}$ and $(\mathcal{E},g^{\mathcal{E}},\nabla^{\mathcal{E}})|_{Y(\nu)}$, respectively.

Let $D^{Z_{[\mu,\nu]}}$ be the restriction of $D^Y$ to $Z_{[\mu,\nu]}$, and
$D^{Z_{[\mu,\nu]}}_{g}$ be the conjugation of $D^{Z_{[\mu,\nu]}}$ by $g$ defined as in \eqref{conju}.
If we impose the local boundary condition which is defined by
families of self-adjoint and invertible endomorphisms $F_1$ of $E|_{Z\times\{\mu\}}$
and $F_2$ of $E|_{Z\times\{\nu\}}$ on the curve
\begin{equation}\label{y21}
D^{Z_{[\mu,\nu]}}_u=(1-u)D^{Z_{[\mu,\nu]}}+uD^{Z_{[\mu,\nu]}}_{g},\ u\in [0,1],
\end{equation}
then as in Section \ref{sec2.4}, we have a higher spectral flow
\begin{equation*}
\SF\Big\{
D^{Z_{[\mu,\nu]}}_{\{F_1,F_2\}},D^{Z_{[\mu,\nu]}}_{g_1,\{F_1,F_2\}}
\Big\}
\end{equation*}
associated to the curve $D^{Z_{[\mu,\nu]}}_{u,\{F_1,F_2\}}$, $u\in [0,1]$ of realizations of \eqref{y21},
provided that the $K^1$-index of $D^{Z_{[\mu,\nu]}}_{\{F_1,F_2\}}$ vanishes.

Recall that $\widetilde{F}={\rm id}|_{E_+}\oplus (-{\rm id}|_{E_-})$.

The following splitting formula for the higher spectral flow enables us to
reduce the problem to the computation on a family of finite cylinders.

\begin{prop}\label{y20}
The following identity holds in $K(S)$,
\begin{equation}
\begin{split}\label{sp1}
\SF\{D^Y_F,D^Y_{g,F}\}
&=\SF\Big\{
D^{Z_{[0,6\epsilon]}}_{\{\widetilde{F},\,{\rm id}\}},D^{Z_{[0,6\epsilon]}}_{g,\{\widetilde{F},\,{\rm id}\}}
\Big\}
\\
&=\SF\Big\{
D^{Z_{[0,6\epsilon]}}_{\{\widetilde{F},-{\rm id}\}},D^{Z_{[0,6\epsilon]}}_{g,\{\widetilde{F},-{\rm id}\}}
\Big\}.
\end{split}
\end{equation}
\end{prop}
\begin{proof}

Let $\xi$ be a smooth function on $[2\epsilon,6\epsilon]$ such that
\begin{equation*}
\xi|_{[2\epsilon,3\epsilon]}=1,\quad 0\leq\xi|_{[3\epsilon,5\epsilon]}\leq 1,\quad \xi|_{[5\epsilon,6\epsilon]}=0
\end{equation*}
and that
$\gamma=(1-\xi^2)^{1/2}$
is also smooth. We extend $\xi$ to $Z_{[0,6\epsilon]}$ by
equaling $1$ on $Z_{[0,2\epsilon]}$ and to
$Y(2\epsilon)$ by equaling $0$ on $Y(6\epsilon)$. Thus $\gamma$ also extends to $Z_{[0,6\epsilon]}$ and $Y(2\epsilon)$.

As in \cite[(3.11)]{MR1447295}, we consider the Hilbert bundles
\begin{equation}
\begin{split}
H_0&=\textbf{H}^0(Y,S(TY)\otimes E)\oplus \textbf{H}^0(Z_{[2\epsilon,6\epsilon]},S(TY)\otimes E),
\\
H_1&=\textbf{H}^0(Z_{[0,6\epsilon]},S(TY)\otimes E)\oplus \textbf{H}^0(Y(2\epsilon),S(TY)\otimes E).
\end{split}
\end{equation}

Following Bunke \cite{MR1331973} (see also \cite{MR1447295}),
we define the operators $a$, $b$, $c$, $d:H_0\rightarrow H_1$
fiberwise by
\begin{enumerate}[{(\romannumeral 1)}]
\item $a$ is the multiplication by $\xi$ on $Y$ followed by the transfer to $Z_{[0,6\epsilon]}$,
\item $b$ is the multiplication by $\xi$ on $Z_{[2\epsilon,6\epsilon]}$ followed by the transfer to $Y(2\epsilon)$,
\item $c$ is the multiplication by $\gamma$ on $Y$ followed by the transfer to $Y(2\epsilon)$,
\item $d$ is the multiplication by $\gamma$ on $Z_{[2\epsilon,6\epsilon]}$ followed by the transfer to $Z_{[0,6\epsilon]}$.
\end{enumerate}
Here the operator $b$ should not be confused with a point in the base manifold.

Set as in \cite{MR1331973},
\begin{equation}
W=a+b+c-d: H_0\rightarrow H_1.
\end{equation}

\begin{lem}[{\cite{MR1331973}}]
The operator $W: H_0\rightarrow H_1$ is unitary.
\end{lem}

Let $D^{Y(2\epsilon)}$ be the restriction of $D^Y$ to $Y(2\epsilon)$, and
$D_{{\rm id}}^{Y(2\epsilon)}$ be the realization of $D^{Y(2\epsilon)}$ imposed with the local
boundary condition defined by the identity endomorphism of $E|_{Z\times\{2\epsilon\}}$.

Set as in \cite{MR1331973} (see also \cite[(3.15)-(3.16)]{MR1447295}),
\begin{equation}\label{d0d1}
\mathcal{D}=D^Y_{\widetilde{F}}\oplus D^{Z_{[2\epsilon,6\epsilon]}}_{\{{\rm id},\,{\rm id}\}}\,,
\ \mathcal{D}'=W^*\Big(D^{Z_{[0,6\epsilon]}}_{\{\widetilde{F},\,{\rm id}\}}\oplus D_{{\rm id}}^{Y(2\epsilon)}\Big)W.
\end{equation}
Then $\mathcal{D}$ and $\mathcal{D}'$ have the same domain and furthermore one verifies easily that the operator
(see \cite{MR1331973} and \cite[(3.17)]{MR1447295})
\begin{equation}\label{yy2}
G=\mathcal{D}'-\mathcal{D}
\end{equation}
is a family of differential operators of zeroth order. In particular, due to Basic Assumption \ref{basic}, it follows that the index bundles
of $\mathcal{D}$ and $\mathcal{D}'$ vanish in $K^1(S)$. Consequently,
the index bundle of $D^{Z_{[0,6\epsilon]}}_{\{\widetilde{F},\,{\rm id}\}}$ vanishes in $K^1(S)$.

Let $\mathcal{D}_{g}$ and $\mathcal{D}'_{g}$ be the conjugations of $\mathcal{D}$ and $\mathcal{D}'$ by $g$ defined as in \eqref{conju}, respectively.
Applying Proposition \ref{hoto}, we see that
\begin{equation}\label{yu}
\SF\{\mathcal{D},\mathcal{D}_{g}\}=\SF\{\mathcal{D}',\mathcal{D}'_{g}\}\quad \text{in}\ K(S).
\end{equation}

Observe that $g$ regarded as a fiberwise bounded operator commutes with $W$.
From \eqref{d0d1}, \eqref{yu} and Proposition \ref{y7}, we get
\begin{equation}\label{sp3}
\SF\{D^Y_{\widetilde{F}},D^Y_{g,\widetilde{F}}\}
=\SF\Big\{
D^{Z_{[0,6\epsilon]}}_{\{\widetilde{F},\,{\rm id}\}},D^{Z_{[0,6\epsilon]}}_{g,\{\widetilde{F},\,{\rm id}\}}
\Big\}\quad \text{in}\ K(S).
\end{equation}
Combining with Proposition \ref{hoto}, we obtain the first equality in \eqref{sp1}.

By replacing the boundary condition ${\rm id}$ in \eqref{d0d1}-\eqref{sp3} by $-{\rm id}$ and
using the same argument, we get
the second equality in \eqref{sp1}.
\end{proof}

\subsection{Computation on finite cylinders}\label{sec3.3}

\begin{thm}\label{cyl}
The following identities hold in $K(S)$,
\begin{equation}
\begin{split}\label{yy1}
\SF\Big\{
D^{Z_{[0,6\epsilon]}}_{\{\widetilde{F},\,{\rm id}\}},D^{Z_{[0,6\epsilon]}}_{g,\{\widetilde{F},\,{\rm id}\}}
\Big\}
&=-\SF\{B^Z_{-},B^Z_{-,g|_{\partial M}}\},
\\
\SF\Big\{
D^{Z_{[0,6\epsilon]}}_{\{\widetilde{F},-{\rm id}\}},D^{Z_{[0,6\epsilon]}}_{g,\{\widetilde{F},-{\rm id}\}}
\Big\}
&=\SF\{B^Z_{+},B^Z_{+,g|_{\partial M}}\}.
\end{split}
\end{equation}
\end{thm}
\begin{proof}
According to \eqref{bdeq}, \eqref{bdeq2} and the definition of $D^{Z_{[0,6\epsilon]}}$, we can
split $D^{Z_{[0,6\epsilon]}}$ with respect to the splitting $E|_Z=E_+\oplus E_-$ as
\begin{equation}\label{y11}
D^{Z_{[0,6\epsilon]}}=D^{Z_{[0,6\epsilon]}}_+\oplus D^{Z_{[0,6\epsilon]}}_-.
\end{equation}
Let $D^{Z_{[0,6\epsilon]}}_{+,g}$ and $D^{Z_{[0,6\epsilon]}}_{-,g}$ be the conjugations of
$D^{Z_{[0,6\epsilon]}}_{+}$ and $D^{Z_{[0,6\epsilon]}}_{-}$ by $g$ defined as in \eqref{conju}, respectively.

As in \eqref{toep}, we consider the curves
\begin{equation}\label{y12}
D^{Z_{[0,6\epsilon]}}_{\pm,u}=(1-u)D^{Z_{[0,6\epsilon]}}_{\pm}+u D^{Z_{[0,6\epsilon]}}_{\pm,g},\ u\in [0,1].
\end{equation}
We impose on $D^{Z_{[0,6\epsilon]}}_{+,u}$, $u\in [0,1]$ the local boundary condition which is defined by the endomorphisms
${\rm id}|_{E_+}$ of $E_+|_{Z}$ and ${\rm id}|_{E_+}$ of $E_+|_{Z\times \{6\epsilon\}}$, and impose on
$D^{Z_{[0,6\epsilon]}}_{-,u}$, $u\in [0,1]$ the local boundary condition which is defined by the endomorphisms
$-{\rm id}|_{E_-}$ of $E_-|_{Z}$ and ${\rm id}|_{E_-}$ of $E_-|_{Z\times \{6\epsilon\}}$.
Let $D^{Z_{[0,6\epsilon]}}_{+,u,\{{\rm id},\,{\rm id}\}}$
and $D^{Z_{[0,6\epsilon]}}_{-,u,\{-{\rm id},\,{\rm id}\}}$ denote the corresponding realizations.
Then by \eqref{y21}, \eqref{y11} and \eqref{y12},
we have the splitting of the curves
\begin{equation}\label{yy3}
D^{Z_{[0,6\epsilon]}}_{u,\{\widetilde{F},\,{\rm id}\}}
=D^{Z_{[0,6\epsilon]}}_{+,u,\{{\rm id},\,{\rm id}\}}\oplus
D^{Z_{[0,6\epsilon]}}_{-,u,\{-{\rm id},\,{\rm id}\}},\ u\in [0,1].
\end{equation}

Due to the vanishing of the $K^1$-index of $D^{Z_{[0,6\epsilon]}}_{\{\widetilde{F},\,{\rm id}\}}$ and Proposition \ref{y7},
we obtain the vanishing of the $K^1$-index of $D^{Z_{[0,6\epsilon]}}_{-,\{-{\rm id},\,{\rm id}\}}$ and
the splitting of the higher spectral flows associated to \eqref{yy3}, which implies
\begin{equation}\label{split}
\SF\Big\{
D^{Z_{[0,6\epsilon]}}_{\{\widetilde{F},\,{\rm id}\}},D^{Z_{[0,6\epsilon]}}_{g,\{\widetilde{F},\,{\rm id}\}}
\Big\}
=
\SF\Big\{
D^{Z_{[0,6\epsilon]}}_{-,\{-{\rm id},\,{\rm id}\}},D^{Z_{[0,6\epsilon]}}_{-,g,\{-{\rm id},\,{\rm id}\}}
\Big\}.
\end{equation}

We proceed to compute the term on the right hand side of \eqref{split}.

Using \eqref{y14}, \eqref{y43}, \eqref{bdeq2}, \eqref{y11} and \eqref{y12}, we regard $D^{Z_{[0,6\epsilon]}}_{-,u}$, $u\in [0,1]$
as families of operators acting on
$\big(S_+(TY)\otimes E\oplus S_+(TY)\otimes E\big)\big|_{Z_{[0,6\epsilon]}}$ like
\begin{equation}\label{y17}
\begin{bmatrix}
0 &-\frac{\partial}{\partial r}+B^Z_{-,u}\,
\\ \,\frac{\partial}{\partial r}+B^Z_{-,u} &0
\end{bmatrix},\ u\in [0,1].
\end{equation}
Moreover, the local boundary condition imposed on
$D^{Z_{[0,6\epsilon]}}_{-,u}$, $u\in [0,1]$ takes the form under identification \eqref{identi}:                                                                                                \begin{equation}
\begin{split}\label{y15}
\Bigl\{s_+\oplus s_-\,&\Big|\,s_\pm\in \mathbf{H}^1(Z_{[0,6\epsilon]},S_+(TY)\otimes E),\\
&\hspace{30pt} s_-|_{Z}=-s_+|_{Z},\ s_-|_{Z\times\{6\epsilon\}}=-s_+|_{Z\times\{6\epsilon\}}\Bigr\}.
\end{split}
\end{equation}

For $s\in \textbf{H}^1(Z,S_+(TY)\otimes E_-)$, set
\begin{equation}\label{y16}
\widetilde{s}(r,z)=s(z)\oplus (-Js(z)),\ \text{for}\ r\in [0,6\epsilon], z\in Z.
\end{equation}
Then by \eqref{y15} and \eqref{y16}, we have a naturally defined map
\begin{equation}\label{y19}
\imath: \textbf{H}^1(Z,S_+(TY)\otimes E_-)\longrightarrow \dom\Big(D^{Z_{[0,6\epsilon]}}_{-,u,\{-{\rm id},\,{\rm id}\}}\Big),
\quad
s\longmapsto \widetilde{s}.
\end{equation}

In what follows, for an operator $\mathcal{A}$, we will denote by $\spec(\mathcal{A})$ the spectrum of $\mathcal{A}$.
The computations below are mainly inspired by \cite[Section 6]{MR3077920}.

We fix any $u\in [0,1]$. Observe that $B^Z_{-,u}$ is a family of elliptic self-adjoint differential operators of first order.
By \cite[Theorem {\uppercase\expandafter{\romannumeral 3}}.5.8]{MR1031992}, we know that the fiberwise spectrum $\spec(B^Z_{-,u})$
is discrete and each eigenvalue has finite multiplicity.
Let $\{\lambda_i(u)\}_{i=1}^\infty$ denote all the fiberwise eigenvalues of $B^Z_{-,u}$ which are counted with multiplicities.
Then we can pick a fiberwise $L^2$-orthonormal basis $\{s_i(u)\}_{i=1}^{\infty}$
of $\textbf{H}^0(Z,S_+(TY)\otimes E_-)$, which consists of fiberwise smooth sections
and is such that
\begin{equation*}
B^Z_{-,u}s_i(u)=\lambda_i(u)s_i(u),\ \text{for}\ i\in \mathbb{N}.
\end{equation*}

As explained in Section \ref{sec2.1}, the fiberwise
spectrum $\spec\Big(D^{Z_{[0,6\epsilon]}}_{-,u,\{-{\rm id},\,{\rm id}\}}\Big)$ is discrete and
each eigenvalue has finite multiplicity.
We denote by $H_\lambda$ the fiberwise $\lambda$-eigenspace for any
$\lambda\in \spec\Big(D^{Z_{[0,6\epsilon]}}_{-,u,\{-{\rm id},\,{\rm id}\}}\Big)$.

For any $\lambda\in \spec\Big(D^{Z_{[0,6\epsilon]}}_{-,u,\{-{\rm id},\,{\rm id}\}}\Big)$,
$s=s_+\oplus s_-\in H_\lambda$, we write
\begin{equation}\label{y18}
s_+\oplus (-J s_-)  =\sum_{i=0}^\infty f_i(r)s_i(u)\oplus \sum_{i=0}^\infty \widetilde{f}_i(r)s_i(u).
\end{equation}
Using \eqref{y17}, \eqref{y15} and \eqref{y18}, we infer from
$D^{Z_{[0,6\epsilon]}}_{-,u,\{-{\rm id},\,{\rm id}\}}s=\lambda s$ that
\begin{equation}
\begin{split}\label{loceq}
\tfrac{\partial}{\partial r}f_i+\lambda_i f_i
&=\lambda \widetilde{f}_i,\
-\tfrac{\partial}{\partial r}\widetilde{f}_i+\lambda_i \widetilde{f}_i=\lambda f_i,
\\
&f_i(r)+\widetilde{f}_i(r)=0,\ \text{when}\ r=0,6\epsilon.
\end{split}
\end{equation}
From the first line in \eqref{loceq}, we deduce that
\begin{equation}
\frac{\partial}{\partial r^2}(f_i+\widetilde{f}_i)=(\lambda_i^2-\lambda^2)(f_i+\widetilde{f}_i).
\end{equation}
Therefore, we have the following cases:
\begin{enumerate}[{(\romannumeral 1)}]
\item either $f_i(r)=\widetilde{f}_i(r)=0$,
\item or $\lambda=-\lambda_i$, $f_i(r)={\rm const}$, and $\widetilde{f}_i(r)=-f_i(r)$,
\item or $\lambda_i^2-\lambda^2=-\frac{n^2\pi^2}{36\epsilon^2}$, $n\in \mathbb{Z}\setminus\{0\}$,
\begin{equation*}
f_i(r)+\widetilde{f}_i(r)={\rm const}\cdot \big(\exp{\tfrac{\sqrt{-1}n\pi r}{6\epsilon}}-\exp(-\tfrac{\sqrt{-1}n\pi r}{6\epsilon})\big).
\end{equation*}
\end{enumerate}

Denote $\spec^{\rm c}=\spec\Big(D^{Z_{[0,6\epsilon]}}_{-,u,\{-{\rm id},\,{\rm id}\}}\Big)\Big\backslash\spec\,(-B^Z_{-,u})$.
Then from the discussions above, we have the fiberwise $L^2$-orthogonal decomposition
\begin{equation}\label{y4}
\dom\Big(D^{Z_{[0,6\epsilon]}}_{-,u,\{-{\rm id},\,{\rm id}\}}\Big)
=\bigoplus_{\lambda\in \spec(-B^Z_{-,u})}H_\lambda\oplus\bigoplus_{\lambda\in \spec^{\rm c}}H_\lambda.
\end{equation}
Moreover, by case {(\romannumeral 2)} and \eqref{y16}-\eqref{y18}, we know that
\begin{equation}
H_\lambda=\imath(\mathcal{H}_{-\lambda}),\ \text{for}\ \lambda\in \spec(-B^Z_{-,u}),
\end{equation}
where by $\mathcal{H}_\nu$ we mean the $\nu$-eigenspace for $\nu\in \spec\,(B^Z_{-,u})$. Therefore,
\begin{equation}\label{y8}
\bigoplus_{\lambda\in \spec(-B^Z_{-,u})}H_\lambda
=\imath\Big(\,\bigoplus_{\lambda\in \spec(B^Z_{-,u})}\mathcal{H}_\lambda\,\Big)
=\imath\Big(\textbf{H}^1(Z,S_+(TY)\otimes E_-)\Big)
\end{equation}

Denote $H^\perp=\bigoplus_{\lambda\in \spec^{\rm c}}H_\lambda$. From \eqref{y4} and \eqref{y8}, we get
the fiberwise $L^2$-orthogonal decomposition
\begin{equation}\label{y9}
\dom\Big(D^{Z_{[0,6\epsilon]}}_{-,u,\{-{\rm id},\,{\rm id}\}}\Big)
=\imath\Big(\textbf{H}^1(Z,S_+(TY)\otimes E_-)\Big)\oplus H^\perp.
\end{equation}
We emphasize here that the decomposition \eqref{y9} does not depend on the choice of $\{s_i(u)\}_{i=1}^{\infty}$.

Corresponding to \eqref{y9}, we can split $D^{Z_{[0,6\epsilon]}}_{-,u,\{-{\rm id},\,{\rm id}\}}$ intrinsically as
\begin{equation}\label{y5}
D^{Z_{[0,6\epsilon]}}_{-,u,\{-{\rm id},\,{\rm id}\}}=\big(\!-B^Z_{-,u}\big)\oplus \mathcal{A}_u\,,
\end{equation}
where $\mathcal{A}_u$ is a family of invertible operators.
Varying the parameter $u$ in $[0,1]$, we get the splitting of the curves associated to \eqref{y5}.

From \eqref{y5}, we infer that
\begin{equation}
\ind\big(D^{Z_{[0,6\epsilon]}}_{-,\{-{\rm id},\,{\rm id}\}}\big)=-\ind(B^Z_-)\ \ \text{in}\ K^1(S).
\end{equation}
Consequently, the spectral section for $B^Z_-$ always exists.
Moreover,  we conclude that
\begin{equation}\label{cybd}
\SF\Big\{
D^{Z_{[0,6\epsilon]}}_{-,\{-{\rm id},\,{\rm id}\}},D^{Z_{[0,6\epsilon]}}_{-,g,\{-{\rm id},\,{\rm id}\}}
\Big\}
=-\SF\{B^Z_{-},B^Z_{-,g|_{\partial M}}\}.
\end{equation}

Combining \eqref{split} with \eqref{cybd}, we get the first line in \eqref{yy1}. The proof of the second line
in \eqref{yy1} follows in a similar manner.
\end{proof}

Combining Proposition \ref{y20} with Theorem \ref{cyl}, we get Theorem \ref{main}.


\begin{thebibliography}{10}

\bibitem{MR0397799}
M.~F. Atiyah, V.~K. Patodi, and I.~M. Singer.
\newblock Spectral asymmetry and {R}iemannian geometry. {III}.
\newblock {\em Math. Proc. Cambridge Philos. Soc.}, 79(1):71--99, 1976.

\bibitem{MR0279833}
M.~F. Atiyah and I.~M. Singer.
\newblock The index of elliptic operators. {IV}.
\newblock {\em Ann. of Math. (2)}, 93:119--138, 1971.

\bibitem{MR1262254}
J.~Avron, R.~Seiler, and B.~Simon.
\newblock The index of a pair of projections.
\newblock {\em J. Funct. Anal.}, 120(1):220--237, 1994.

\bibitem{MR3034288}
Moulay-Tahar Benameur and Alan~L. Carey.
\newblock Higher spectral flow and an entire bivariant {JLO} cocycle.
\newblock {\em J. K-Theory}, 11(1):183--232, 2013.

\bibitem{MR813584}
Jean-Michel Bismut.
\newblock The {A}tiyah-{S}inger index theorem for families of {D}irac
  operators: two heat equation proofs.
\newblock {\em Invent. Math.}, 83(1):91--151, 1985.

\bibitem{MR2536846}
Bernhelm Boo{\ss}-Bavnbek, Matthias Lesch, and Chaofeng Zhu.
\newblock The {C}alder\'on projection: new definition and applications.
\newblock {\em J. Geom. Phys.}, 59(7):784--826, 2009.

\bibitem{MR1233386}
Bernhelm Boo{\ss}-Bavnbek and Krzysztof~P. Wojciechowski.
\newblock {\em Elliptic {B}oundary {P}roblems for {D}irac {O}perators}.
\newblock Mathematics: Theory \& Applications. Birkh\"auser Boston Inc.,
  Boston, MA, 1993.

\bibitem{MR1331973}
Ulrich Bunke.
\newblock On the gluing problem for the {$\eta$}-invariant.
\newblock {\em J. Differential Geom.}, 41(2):397--448, 1995.

\bibitem{MR1393386}
Xianzhe Dai and Weiping Zhang.
\newblock Higher spectral flow.
\newblock {\em Math. Res. Lett.}, 3(1):93--102, 1996.

\bibitem{MR1447295}
Xianzhe Dai and Weiping Zhang.
\newblock Splitting of the family index.
\newblock {\em Comm. Math. Phys.}, 182(2):303--317, 1996.

\bibitem{MR1638328}
Xianzhe Dai and Weiping Zhang.
\newblock Higher spectral flow.
\newblock {\em J. Funct. Anal.}, 157(2):432--469, 1998.

\bibitem{MR1651419}
Daniel~S. Freed.
\newblock Two index theorems in odd dimensions.
\newblock {\em Comm. Anal. Geom.}, 6(2):317--329, 1998.

\bibitem{MR1231957}
Ezra Getzler.
\newblock The odd {C}hern character in cyclic homology and spectral flow.
\newblock {\em Topology}, 32(3):489--507, 1993.

\bibitem{gl}
Alexander Gorokhovsky and Matthias Lesch.
\newblock On the spectral flow for {D}irac operators with local boundary
  conditions.
\\ See http://arxiv.org/pdf/1310.0210v2.pdf.

\bibitem{MR1335452}
Tosio Kato.
\newblock {\em Perturbation {T}heory for {L}inear {O}perators}.
\newblock Classics in Mathematics. Springer-Verlag, Berlin, 1995.
\newblock Reprint of the 1980 edition.

\bibitem{MR1031992}
H.~Blaine Lawson, Jr. and Marie-Louise Michelsohn.
\newblock {\em Spin {G}eometry}, volume~38 of {\em Princeton Mathematical
  Series}.
\newblock Princeton University Press, Princeton, NJ, 1989.

\bibitem{MR1601842}
E.~Leichtnam and P.~Piazza.
\newblock Spectral sections and higher {A}tiyah-{P}atodi-{S}inger index theory
  on {G}alois coverings.
\newblock {\em Geom. Funct. Anal.}, 8(1):17--58, 1998.

\bibitem{MR1744589}
Eric Leichtnam and Paolo Piazza.
\newblock A higher {A}tiyah-{P}atodi-{S}inger index theorem for the signature
  operator on {G}alois coverings.
\newblock {\em Ann. Global Anal. Geom.}, 18(2):171--189, 2000.

\bibitem{MR1979016}
Eric Leichtnam and Paolo Piazza.
\newblock Dirac index classes and the noncommutative spectral flow.
\newblock {\em J. Funct. Anal.}, 200(2):348--400, 2003.

\bibitem{MR2114489}
Matthias Lesch.
\newblock The uniqueness of the spectral flow on spaces of unbounded
  self-adjoint {F}redholm operators.
\newblock In {\em Spectral {G}eometry of {M}anifolds with {B}oundary and
  {D}ecomposition of {M}anifolds}, volume 366 of {\em Contemp. Math.}, pages
  193--224. Amer. Math. Soc., Providence, RI, 2005.

\bibitem{MR1472895}
Richard~B. Melrose and Paolo Piazza.
\newblock Families of {D}irac operators, boundaries and the {$b$}-calculus.
\newblock {\em J. Differential Geom.}, 46(1):99--180, 1997.

\bibitem{MR3077920}
Marina Prokhorova.
\newblock The spectral flow for {D}irac operators on compact planar domains
  with local boundary conditions.
\newblock {\em Comm. Math. Phys.}, 322(2):385--414, 2013.

\bibitem{MR1864735}
Weiping Zhang.
\newblock {\em Lectures on {C}hern-{W}eil {T}heory and {W}itten
  {D}eformations}, volume~4 of {\em Nankai Tracts in Mathematics}.
\newblock World Scientific Publishing Co. Inc., River Edge, NJ, 2001.

\end{thebibliography}

\end{document}